\documentclass[11pt]{amsart}

\usepackage{amsmath}
\usepackage{amsthm}
\usepackage{amsfonts}
\usepackage{amssymb}

\usepackage{graphics}
\usepackage{epsfig}
\usepackage{color}
\usepackage{verbatim}
\newcommand\hess{\operatorname{Hess}}

\newcommand\R{{\mathbb{R}}}

\newcommand\supp{\operatorname{supp}}

\theoremstyle{plain}
  \newtheorem{theorem}{Theorem}[section]

  \newtheorem{proposition}[theorem]{Proposition}
  
  \newtheorem{lemma}[theorem]{Lemma}
  \newtheorem{corollary}[theorem]{Corollary}

\theoremstyle{definition}

  \numberwithin{equation}{section}
\title[A multilinear Fourier extension identity]{A multilinear Fourier extension identity on $\mathbb{R}^n$}
\author{Jonathan Bennett and
Marina Iliopoulou}
\thanks{This work was supported by the European Research Council [grant
number 307617]}

\address{Jonathan Bennett: School of Mathematics, The Watson Building, University of Birmingham, Edgbaston,
Birmingham, B15 2TT, England.}
\email{J.Bennett@bham.ac.uk}
\address{Marina Iliopoulou:  Department of Mathematics, University of California, Berkeley, CA 94720-3840, USA
}
\email{m.iliopoulou@berkeley.edu}

\begin{document}
\begin{abstract}
We prove an elementary multilinear identity for the Fourier extension operator on $\mathbb{R}^n$, generalising to higher dimensions the classical bilinear extension identity in the plane. In the particular case of the extension operator associated with the paraboloid, this provides a higher dimensional extension of a well-known identity of Ozawa and Tsutsumi for solutions to the free time-dependent Schr\"odinger equation. We conclude with a similar treatment of more general oscillatory integral operators whose phase functions collectively satisfy a natural multilinear transversality condition. The perspective we present has its origins in work of Drury.
\end{abstract}
\maketitle
\section{Introduction}
To a smooth function $\phi:\mathbb{R}^{n-1}\rightarrow\mathbb{R}$ we associate the \emph{Fourier extension operator}
$$
Eg(x)=\int_{\mathbb{R}^{n-1}}e^{i(x'\cdot\xi+x_n\phi(\xi))}g(\xi)d\xi;$$
here $x=(x',x_n)\in\mathbb{R}^{n-1}\times\mathbb{R}$, and a-priori $g\in L^1(\mathbb{R}^{n-1})$.
The term ``extension operator" is used since the adjoint $E^*$, given by $E^*f(\xi)=\widehat{f}(\xi,\phi(\xi))$, gives a (parametrised) restriction of the Fourier transform of a function $f$ on $\mathbb{R}^n$ to the hypersurface $S=\{(\xi,\phi(\xi)): \xi\in\mathbb{R}^{n-1}\}$. In practice the function $\phi$ is often only defined on some compact set $U\subseteq\mathbb{R}^{n-1}$, giving rise to a compact hypersurface  $S$. We gloss over this point in most of what follows since such a feature may be captured by the implicit assertion that the function $g$ is supported in $U$. In the 1960s Stein observed that if $S$ is compact and has everywhere nonvanishing curvature, then $E$ satisfies estimates of the form 
\begin{equation}\label{restcon}
\|Eg\|_{L^q(\mathbb{R}^n)}\lesssim\|g\|_{L^p(U)}
\end{equation} 
with $q<\infty$; the case $(p,q)=(1,\infty)$ is of course elementary by Minkowski's inequality. The celebrated \textit{Fourier restriction conjecture} asserts that estimates of this type continue to hold for $q>\frac{2n}{n-1}$, with elementary examples preventing an endpoint estimate at $q=\frac{2n}{n-1}$; see for example \cite{S}. Since the 1990s \textit{bilinear}, and more generally \textit{multilinear}, estimates of this type have emerged as particularly natural and useful; see for example \cite{TVV}, \cite{TaoSurvey}, \cite{Mattila}, \cite{BCT}, \cite{BEsc}, \cite{BDAnnals}. The simplest such example is the well-known and elementary bilinear identity
\begin{equation}\label{bilin}
\int_{\mathbb{R}^2}|E_1g_1(x)E_2g_2(x)|^2dx=(2\pi)^{2}\int_{\mathbb{R}^2}\frac{|g_1(\xi_1)|^2|g_2(\xi_2)|^2}{|\phi_1'(\xi_1)-\phi_2'(\xi_2)|}d\xi_1d\xi_2,
\end{equation}
where $E_1, E_2$ are extension operators associated with phases $\phi_1,\phi_2$ and curves $S_1, S_2$ in the plane; see \cite{Fefferman1970} for the origins of this.\footnote{As may be expected, some technical hypotheses relating to the geometry of these curves are needed here, and it will suffice to ask that $\phi'_1(\xi_1)\not=\phi'_2(\xi_2)$ whenever $\xi_j$ belongs to some interval containing the support of $g_j$, for each $j=1,2$.}
This particular two-dimensional statement occupies a singular position in Fourier restriction theory in the sense that it is \emph{an identity}. The main purpose of this paper is to establish natural higher-dimensional analogues of this.
To this end we consider extension operators $E_1,\hdots,E_n$ associated with the functions $\phi_1,\hdots,\phi_n$, and hypersurfaces $S_1,\hdots,S_n$.
\begin{theorem}\label{main}
\begin{equation}\label{id}
|E_1g_1|^2*\cdots *|E_ng_n|^2\equiv(2\pi)^{n(n-1)}\int_{(\mathbb{R}^{n-1})^n}\frac{|g_1(\xi_1)|^2\cdots|g_n(\xi_n)|^2}{\left|\det\left(
\begin{array}{ccc}
1& \cdots  & 1\\
\nabla\phi_1(\xi_1) &
\cdots & \nabla\phi_n(\xi_n)\\
\end{array}\right)\right|}d\xi,
\end{equation}
for all functions $g_1,\hdots,g_n$ such that the determinant factor is nonzero whenever $\xi_j$ belongs to the convex hull of the support of $g_j$, $1\leq j\leq n$.
\end{theorem}
It should be remarked that requiring a non-vanishing determinant factor whenever $\xi_j$ belongs to the \textit{support} of $g_j$ ($1\leq j\leq n$) is necessary in order for the integral on the right hand side of \eqref{id} to be finite. This is due to a critical lack of local integrability, which is of course also present in \eqref{bilin}. Our requirement that this continues to hold on the convex hull of the supports is a technical condition used in our proof, and is a product of the generality of the set-up. As we shall see, this is not always necessary, as the particular case where each $S_j$ is the paraboloid reveals. In particular, the following holds.
\begin{theorem} \label{main'} Let $E$ be the extension operator on the paraboloid $S=\{\left(\xi,\phi(\xi)\right):\xi\in\R^{n-1}\}$, with $\phi=|\;\cdot\; |^2$. Then,
\begin{equation}\label{eq:main'eq}
|Eg_1|^2*\cdots *|Eg_n|^2\equiv 2^{-(n-1)}(2\pi)^{n(n-1)}
\int_{(\mathbb{R}^{n-1})^n}\frac{|g_1(\xi_1)|^2\cdots|g_n(\xi_n)|^2}{\left|\det\left(
\begin{array}{ccc}
1& \cdots  & 1\\
\xi_1 &
\cdots & \xi_n\\
\end{array}\right)\right|}d\xi.
\end{equation}
\end{theorem}
We clarify that while Theorem \ref{main'} does not impose a support condition on the functions $g_j$, finiteness in \eqref{eq:main'eq} requires that the determinant factor on the right hand side does not vanish on their supports. This particular determinant factor is of course just the volume of the parallelepiped in $\mathbb{R}^{n-1}$ with vertices $\xi_1,\hdots,\xi_n$.

Theorem \ref{main} tells us that $|E_1g_1|^2*\cdots *|E_ng_n|^2$ is a constant function. Nevertheless, it is enough to prove \eqref{id} at the origin, as the right hand side is manifestly modulation-invariant. The case $n=2$ of Theorem \ref{main}
immediately reduces to \eqref{bilin} on evaluating the convolution at the origin and performing a harmless reflection in either $E_1g_1$ or $E_2g_2$. The identity \eqref{id} may be interpreted as an elementary substitute for the absence of a linear restriction inequality (of the form \eqref{restcon}) at the endpoint $q=2n/(n-1)$.
%$$\|Eg\|_{L^{\frac{2n}{n-1}}(\R^n)}\lesssim \|g\|_{L^{\frac{2n}{n-1}}(U)}
%$$
%in the classical linear restriction conjecture for the extension operator $E$ on a smooth compact hypersurface in $\R^n$ with non-vanishing gaussian curvature. 
Indeed, notice that the $n$-fold convolution $$L^{n/(n-1)}(\mathbb{R}^n)*\cdots *L^{n/(n-1)}(\mathbb{R}^n)\subseteq L^\infty(\mathbb{R}^n)$$ by Young's convolution inequality; therefore, an inequality of the form \eqref{restcon} at $q=2n/(n-1)$ would also imply that $|E_1g_1|^2*\cdots *|E_ng_n|^2$ is a bounded function. This perspective on the restriction conjecture originates in work of Drury, and the underlying ideas in this paper are closely related to those in \cite{Drury}.

%Of course the left hand side of \eqref{id} at the origin may be formally written as
%\begin{equation*}%\label{alt1}
%(2\pi)^{n(n-1)}\int_{\mathbb{R}^n}\prod_{j=1}^n\mathcal{F}^{-1}|E_jg_j|^2,
%\end{equation*}
%where $\mathcal{F}$ denotes the Fourier transform.
A more geometric interpretation of \eqref{id} comes from writing $$E_jg_j=\widehat{f_jd\sigma_j},$$ where $d\sigma_j$ is surface area measure on $S_j$, and $f_j$ is given by $$g_j(\xi)=(1+|\nabla\phi_j(\xi)|^2)^{1/2}f_j(\xi,\phi_j(\xi)).$$ In these terms \eqref{id} becomes
\begin{equation}\label{geo}
%\int_{\mathbb{R}^n}\prod_{j=1}^n (f_jd\sigma_j)*(\widetilde{f_jd\sigma_j})
|\widehat{f_1d\sigma_1}|^2*\cdots *|\widehat{f_nd\sigma_n}|^2\equiv
(2\pi)^{n(n-1)}\int_{S_1\times\cdots\times S_n}\frac{|f_1(y_1)|^2\cdots |f_n(y_n)|^2}{|v_1(y_1)\wedge\cdots\wedge v_n(y_n)|}d\sigma_1(y_1)\cdots d\sigma_n(y_n),
\end{equation}
where
%$\widetilde{\mu}$ denotes the reflection of a measure $\mu$ in the origin, and
$v_j(y_j)$ denotes a unit normal vector to $S_j$ at the point $y_j\in S_j$. It is instructive to (formally) take the Fourier transform of the identity \eqref{geo}, and look to interpret the resulting distribution $$\prod_{j=1}^n (f_jd\sigma_j)*(\widetilde{f_jd\sigma_j})$$ as a multiple of the delta distribution at the origin; here $\widetilde{\mu}$ denotes the reflection of a measure $\mu$ in the origin. The key observation is that each factor $(f_jd\sigma_j)*(\widetilde{f_jd\sigma_j})$ is supported in the complement of a cone with vertex at $0$, and the axes of these cones point in a spanning set of directions. We do not attempt to make these heuristics rigorous here.

%The identity \eqref{geo} may be viewed as a substitute for the false endpoint
%$$\|(fd\sigma)*(\widetilde{fd\sigma})\|_{L^n(\mathbb{R}^n)}\lesssim \|f\|_{L^{\frac{2n}{n-1}}(S)}^2,
%$$
%which would follow from the missing endpoint
%$$\|\widehat{fd\sigma}\|_{L^{\frac{2n}{n-1}}(\mathbb{R}^n)}\lesssim \|f\|_{L^{\frac{2n}{n-1}}(S)}$$
%in the linear restriction conjecture by an application of Young's convolution inequality; note that if functions $h_1,\hdots,h_n\in L^{n/(n-1)}(\mathbb{R}^n)$, then $h_1*\cdots*h_n\in L^\infty$.
%by an application of the Hausdorff--Young inequality; here $S$ is a compact piece of smooth hypersurface with nonvanishing gaussian curvature, such as a portion of the paraboloid or sphere.

%It is instructive to give a direct proof of \eqref{geo}, although we leave this to the interested reader.

We conclude this section with some further contextual remarks and generalisations.

Notice that the vector $(1,-\nabla\phi_j(\xi_j))^T$ is normal to the hypersurface $S_j$ at the point $(\xi_j,\phi_j(\xi_j))$, and so if the surfaces $S_1,\hdots,S_n$ are compact and \emph{transversal}, that is, satisfying\footnote{Throughout this paper we shall write $A\lesssim B$ if there exists a constant $c$ such that $A\leq cB$. The relations $A\gtrsim B$ and $A\sim B$ are defined similarly.} $$|v_1(y_1)\wedge\cdots\wedge v_n(y_n)|\gtrsim 1\;\;\mbox{ for }\;\;y_1\in S_1, \hdots, y_n\in S_n,$$
then \eqref{id} becomes
\begin{equation}\label{alttran}
|E_1g_1|^2*\cdots *|E_ng_n|^2\sim\|g_1\|_2^2\cdots\|g_n\|_2^2.
\end{equation}
%Notice that, unlike in most statements involving extension operators, there is no requirement that the surfaces $S_j$ have curvature in any of these statements. Indeed, if the functions $\phi_j$ are \emph{linear}, then the determinant factor is constant, and the claimed identity \eqref{id} reduces to Plancherel's theorem and Fubini's theorem.
It is interesting to contrast this with the (considerably deeper) endpoint multilinear restriction conjecture
\begin{equation}\label{mlr}
\|E_1g_1\cdots E_ng_n\|_{L^{\frac{2}{n-1}}(\mathbb{R}^n)}\lesssim\|g_1\|_2\cdots\|g_n\|_2;
\end{equation}
see \cite{BCT}.
%, which in this degenerate situation, reduces to Plancherel's theorem and the Loomis--Whitney inequality. %\footnote{It might be interesting to interpolate \eqref{alttran} and \eqref{mlr} (perhaps using analytic interpolation via the distribution $\alpha_s(y)=\frac{e^{s^2}}{\Gamma(s)}y^{s-1}$).}
While \eqref{mlr} remains open, the weaker
\begin{equation}\label{BCTthm}
\|E_1g_1\cdots E_ng_n\|_{L^{\frac{2}{n-1}}(B(0;R))}\lesssim_\varepsilon R^\varepsilon\|g_1\|_2\cdots\|g_n\|_2,\;\;\;R\gg 1,
\end{equation}
is known; see \cite{BCT}, \cite{BEsc} for a modest improvement, and \cite{BBFL}, \cite{Z} for generalisations.

Theorem \ref{main} is a particular case of a one-parameter family of identities for the multilinear operator
$
T_\sigma(g_1,\hdots,g_n)(x_1,\hdots,x_n):=$
$$\int_{(\mathbb{R}^{n-1})^n}\left|\det\left(
\begin{array}{ccc}
1& \cdots  & 1\\
\nabla\phi_1(\xi_1) &
\cdots & \nabla\phi_n(\xi_n)\\
\end{array}\right)\right|^\sigma\prod_{j=1}^n e^{i(x_j'\cdot\xi_j+x_{jn}\phi_j(\xi_j))}g_j(\xi_j)d\xi_j,
$$
where $x_1,\hdots, x_n\in\mathbb{R}^n$, $x_j=(x_j',x_{jn})\in\mathbb{R}^{n-1}\times\mathbb{R}$, and $\sigma\in\mathbb{R}$. Of course $T^0(g_1,\hdots,g_n)=E_1g_1\otimes\cdots \otimes E_ng_n$, so that Theorem \ref{main} is the $\sigma=0$ case of the following:
\begin{theorem}\label{mainagain} For each $\sigma\in\mathbb{R}$,
\begin{eqnarray}\label{idagain}
\begin{aligned}
\int_{x_1+\cdots +x_n=0}&|T_\sigma(g_1,\hdots,g_n)(x_1,\hdots,x_n)|^2dx\\&=(2\pi)^{n(n-1)}\int_{(\mathbb{R}^{n-1})^n}\frac{|g_1(\xi_1)|^2\cdots|g_n(\xi_n)|^2}{\left|\det\left(
\begin{array}{ccc}
1& \cdots  & 1\\
\nabla\phi_1(\xi_1) &
\cdots & \nabla\phi_n(\xi_n)\\
\end{array}\right)\right|^{1-2\sigma}}d\xi
\end{aligned}
\end{eqnarray}
for all functions $g_1,\hdots,g_n$ such that the determinant factor is nonzero whenever $\xi_j$ belongs to the convex hull of the support of $g_j$, $1\leq j\leq n$.
\end{theorem}

In the case of the extension operator on the paraboloid, the support condition on the functions $g_j$ may be dropped provided $\sigma\geq 0$, as our next theorem clarifies.
\begin{theorem}\label{mainagain'} Suppose $\sigma\geq 0$. In the case of the paraboloid, i.e. for $\phi_1=\ldots=\phi_n=\phi=|\;\cdot\;|^2$, 
\begin{equation}\label{idagain'}
\begin{aligned}
\int_{x_1+\cdots +x_n=0}&|T_\sigma(g_1,\hdots,g_n)(x_1,\hdots,x_n)|^2dx\\&=2^{-(n-1)}(2\pi)^{n(n-1)}\int_{(\mathbb{R}^{n-1})^n}\frac{|g_1(\xi_1)|^2\cdots|g_n(\xi_n)|^2}{\left|\det\left(
\begin{array}{ccc}
1& \cdots  & 1\\
\xi_1 &
\cdots & \xi_n\\
\end{array}\right)\right|^{1-2\sigma}}d\xi.
\end{aligned}
\end{equation}
If $\sigma<0$, \eqref{idagain'} continues to hold provided the determinant factor is non-vanishing on the supports of the $g_j$, $1\leq j\leq n$.
\end{theorem}
Of course when $\sigma=0$, Theorem \ref{mainagain'} becomes Theorem \ref{main'}. In contrast with the case $\sigma=0$, when $\sigma>0$ finiteness in \eqref{idagain'} no longer requires that the determinant factor is non-vanishing on the supports of the $g_j$.

Of course \eqref{idagain} ceases to have convolution structure for $\sigma\not=0$. However, alternative geometric insight may be found in a more elementary Kakeya-type analogue of \eqref{idagain}, which states that
\begin{eqnarray}\label{kakeyaid}
\begin{aligned}
\int_{x_1+\cdots +x_n=0}&\Biggl(\sum_{T_1,\hdots,T_n} \left|e(T_1)\wedge\cdots\wedge e(T_n)\right|^{2\sigma}c_{T_1}\chi_{T_1}(x_1)\cdots  c_{T_n}\chi_{T_n}(x_n)\Biggr)\; dx
\\&= c_n\sum_{T_1,\hdots,T_n}\frac{c_{T_1}\cdots c_{T_n}}{\left|e(T_1)\wedge\cdots\wedge e(T_n)\right|^{1-2\sigma}};
\end{aligned}
\end{eqnarray}
here $T_1,\hdots,T_n$ belong to finite sets $\mathbb{T}_1,\hdots,\mathbb{T}_n$
of doubly infinite $1$-tubes (cylinders of cross-sectional volume $1$) in $\mathbb{R}^n$, and for such a tube $T$, $e(T)\in\mathbb{S}^{n-1}$ denotes its direction. Here the coefficients $c_{T_j}$ are nonnegative real numbers, $c_n$ denotes a constant depending only on $n$, and we make the qualitative transversality assumption that $e(T_1)\wedge\cdots\wedge e(T_n)\not=0$ whenever $T_j\in\mathbb{T}_j$. When $n=2$, this is the well-known and elementary bilinear Kakeya theorem in the plane.
By multilinearity \eqref{kakeyaid} immediately follows, for all $\sigma$, from the elementary geometric fact that
\begin{equation}\label{kakeyaidreduced}
%\int_{x_1+\cdots +x_n=0}\chi_{T_1}(x_1)\cdots \chi_{T_n}(x_n)\; dx=
\chi_{T_1}*\cdots *\chi_{T_n}
\equiv\frac{c_n}{\left|e(T_1)\wedge\cdots\wedge e(T_n)\right|}
\end{equation}
whenever $e(T_1)\wedge \cdots\wedge e(T_n)\not=0$. (A simple way to see \eqref{kakeyaidreduced} is to begin with its manifest truth for orthogonal axis-parallel \emph{rectangular} tubes $T_1, \hdots, T_n$, and then use multilinearity and scaling to extend it to orthogonal tubes of arbitrary cross section, whereby a change of variables may then be used to establish the claimed dependence on the directions $e(T_1),\hdots,e(T_n)$.)
The identity \eqref{kakeyaid} with $\sigma=1/2$ has a similar flavour to the much deeper affine-invariant endpoint multilinear Kakeya inequality
$$
\int_{\mathbb{R}^n}\Biggl(\sum_{T_1,\hdots,T_n} \left|e(T_1)\wedge\cdots\wedge e(T_n)\right|c_{T_1}\chi_{T_1}\cdots c_{T_n}\chi_{T_n}\Biggr)^{\frac{1}{n-1}}
\lesssim\Biggl(\sum_{T_1}c_{T_1}\cdots \sum_{T_n}c_{T_n}\Biggr)^{\frac{1}{n-1}}
$$
proved in \cite{BG} and \cite{CV}, and the seemingly deeper still (conjectural) variant
\begin{eqnarray}\label{fuerteKakeya}
\begin{aligned}
\int_{\mathbb{R}^n}\Biggl(\sum_{T_1,\hdots,T_n} &\left|e(T_1)\wedge\cdots\wedge e(T_n)\right|^{2\sigma}c_{T_1}\chi_{T_1}\cdots c_{T_n}\chi_{T_n}\Biggr)^{\frac{1}{n-1}}
\\&\lesssim\Biggl(\sum_{T_1,\hdots,T_n}\frac{c_{T_1}\cdots c_{T_n}}{\left|e(T_1)\wedge\cdots\wedge e(T_n)\right|^{1-2\sigma}}\Biggr)^{\frac{1}{n-1}},
\end{aligned}
\end{eqnarray}
for any real number $\sigma$. This inequality for $\sigma=0$, or at least a natural variant of it involving truncated tubes, is easily seen to imply the classical Kakeya maximal conjecture via an application of Drury's inequalities from \cite{Drury}.
The identities in Theorems \ref{main} and \ref{mainagain}
are inspired by the analogous conjectural multilinear extension inequality
\begin{equation}\label{fuerte}
\|E_1g_1\cdots E_ng_n\|_{L^{\frac{2}{n-1}}(\mathbb{R}^n)}^2\lesssim \int_{(\mathbb{R}^{n-1})^n}\frac{|g_1(\xi_1)|^2\cdots|g_n(\xi_n)|^2}{\left|\det\left(
\begin{array}{ccc}
1& \cdots  & 1\\
\nabla\phi_1(\xi_1) &
\cdots & \nabla\phi_n(\xi_n)\\
\end{array}\right)\right|}d\xi
\end{equation}
and its generalisation
\begin{equation}\label{fuerteagain}
\int_{\mathbb{R}^n}|T_\sigma(g_1,\hdots,g_n)(x,\hdots,x)|^{\frac{2}{n-1}}dx\lesssim \Biggl(\int_{(\mathbb{R}^{n-1})^n}\frac{|g_1(\xi_1)|^2\cdots|g_n(\xi_n)|^2}{\left|\det\left(
\begin{array}{ccc}
1& \cdots  & 1\\
\nabla\phi_1(\xi_1) &
\cdots & \nabla\phi_n(\xi_n)\\
\end{array}\right)\right|^{1-2\sigma}}d\xi\Biggr)^{\frac{1}{n-1}}.
\end{equation}
These very strong conjectural inequalities \eqref{fuerteKakeya}--\eqref{fuerteagain} arose in discussions with Tony Carbery in 2004, and also recall work of Drury in \cite{Drury}.
Some recent progress in this direction may be found in \cite{Ramos}. Of course \eqref{id} and \eqref{idagain} are much more elementary than \eqref{fuerte} and \eqref{fuerteagain} when $n\geq 3$.
%While it is instructive to try to deduce \eqref{id} (in the form of an upper bound) from \eqref{fuerte} using Young's convolution inequality, there appears to be no such superficial link between the two statements when $n\geq 3$.

Theorems \ref{main'} and \ref{mainagain'} may be formulated in terms of solutions $u_1,\hdots,u_{d+1}:\mathbb{R}^d\times\mathbb{R}\rightarrow\mathbb{C}$ to the Schr\"odinger equation $i\partial_tu=\Delta u$ with initial data $f_1,\hdots,f_{d+1}$. Indeed, Theorem \ref{main'} for $n=d+1$ becomes
\begin{eqnarray}\label{ha!}
\begin{aligned}
\int_{\substack{x_1+\cdots +x_{d+1}=0\\
t_1+\cdots +t_{d+1}=0}}&
|u_1(x_1,t_1)|^2\cdots|u_{d+1}(x_{d+1},t_{d+1})|^2dxdt\\&=
\frac{1}{2^{d}(2\pi)^{d(d+1)}}\int_{(\mathbb{R}^{d})^{d+1}}\frac{|\widehat{f}_1(\xi_1)|^2\cdots|\widehat{f}_{d+1}(\xi_{d+1})|^2}{|\rho(\xi)|}d\xi,
\end{aligned}
\end{eqnarray}
where
\begin{displaymath}
\rho(\xi)=\det\left(
\begin{array}{ccc}
1& \cdots& 1\\
\xi_1 &
\cdots & \xi_{d+1}\\
\end{array}\right);
\end{displaymath}
here $\xi=(\xi_1,\hdots,\xi_{d+1})\in\R^d\times\cdots\times\R^d$. 
We observe that $\rho(\xi)=0$ if and only if $\xi_1,\hdots,\xi_{d+1}$ are co-hyperplanar points in $\R^d$, and, in order for the expression in \eqref{ha!} to be finite, one needs to stipulate that the determinant factor is non-vanishing for $\xi_j$ in the support of $\widehat{f}_j$, $1\leq j\leq d+1$. Notice that the tensor product here is a \emph{space-time} tensor product. Thus there are many times in play, and the measure is Lebesgue measure on a linear subspace of space-time. Multilinear expressions of a similar flavour to \eqref{ha!} may be found in \cite{BBFGI}.

A similar reformulation of Theorem \ref{mainagain'} for $\sigma>0$ gives an extension of \eqref{ha!} that ceases to have local integrability (finiteness) issues, retaining content even if the solutions $u_j$ all coincide. In order to state this, it is natural to define the $d$-th order differential operator
$$
\rho(\nabla_x):=\det\left(
\begin{array}{ccc}
1& \cdots  & 1\\
\nabla_{x_1} &
\cdots & \nabla_{x_{d+1}}\\
\end{array}\right),
$$
and its fractional power $|\rho(\nabla_x)|^\gamma$ to be the operator with Fourier multiplier $|\rho(\xi)|^\gamma$; here the Fourier variable $\xi$ belongs to $\mathbb{R}^{d(d+1)}$. In this notation, Theorem \ref{mainagain'} for $\sigma\geq 0$ becomes
\begin{theorem}\label{paraboloidcase}
For solutions $u_1,\hdots,u_{d+1}$ of the Schr\"odinger equation, with initial data $f_1,\hdots,f_{d+1}$ respectively, and for all $\sigma\geq 0$,
$$
\begin{aligned}
\int_{\substack{x_1+\cdots +x_{d+1}=0\\
t_1+\cdots +t_{d+1}=0}}
||\rho(\nabla_x)|^{\sigma}&(u_1(x_1,t_1)\cdots u_{d+1}(x_{d+1},t_{d+1}))|^2dxdt\\&=\frac{1}{2^d(2\pi)^{d(d+1)}}\int_{(\mathbb{R}^{d})^{d+1}}\frac{|\widehat{f}_1(\xi_1)|^2\cdots|\widehat{f}_{d+1}(\xi_{d+1})|^2}{|\rho(\xi)|^{1-2\sigma}}d\xi.
\end{aligned}
$$
\end{theorem}
Setting $\sigma=\frac{1}{2}$ is particularly natural, as it reduces to the following:
 %Thus Theorem \ref{mainagain} with $\sigma=1/2$ provides a natural extension of \eqref{ot2} to higher dimensions.
\begin{corollary}\label{ozts}
\begin{equation}\label{otd}
%\begin{aligned}
\int_{\substack{x_1+\cdots +x_{d+1}=0\\
t_1+\cdots +t_{d+1}=0}}
||\rho(\nabla_x)|^{1/2}(u_1(x_1,t_1)\cdots u_{d+1}(x_{d+1},t_{d+1}))|^2dxdt=\frac{1}{2^{d}}\|f_1\|_2^2\cdots\|f_{d+1}\|_2^2.
%\end{aligned}
\end{equation}
\end{corollary}
The case $d=1$ of Corollary \ref{ozts} is
due to Ozawa and Tsutsumi \cite{OT}, and is more usually stated as
\begin{equation}\label{ot2}
\int_{\mathbb{R}}\int_{\mathbb{R}}|D_x^{1/2}(u_1\overline{u_2})(x,t)|^2dxdt=\frac{1}{2}\|f_1\|^2_2\|f_2\|_2^2,
\end{equation}
where $D_x$ denotes the scalar derivative operator with Fourier multiplier $|\xi|$.
Notice that the complex conjugate and fractional derivative appearing here are encoded in the space-time reflection resulting from the restriction $x_1+x_2=t_1+t_2=0$ in \eqref{otd}. Bilinear extensions of \eqref{ot2} to higher dimensions are also natural, although these cease to be identities; see \cite{BBJP} for further discussion.

As our proof of Theorem \ref{paraboloidcase} reveals, the $\sigma=1$ case may be formulated as
\begin{eqnarray}\label{pvd}
\begin{aligned}
\int_{\substack{x_1+\cdots +x_{d+1}=0\\
t_1+\cdots +t_{d+1}=0}}&
|\rho(\nabla_x)(u_1(x_1,t_1)\cdots u_{d+1}(x_{d+1},t_{d+1}))|^2dxdt\\&=\frac{1}{2^d(2\pi)^{d(d+1)}}\int_{(\mathbb{R}^{d})^{d+1}}|\widehat{f}_1(\xi_1)|^2\cdots|\widehat{f}_{d+1}(\xi_{d+1})|^2|\rho(\xi)|d\xi,
\end{aligned}
\end{eqnarray}
making it somewhat special since it involves only classical derivatives of the solutions. In \cite{PV} (see also \cite{V}), it was shown how to deduce the classical $d=1$ case of \eqref{pvd} from certain bilinear virial identities, avoiding explicit reference to the $u_j$ as Fourier extension operators. This convexity-based approach has the noteworthy advantage of applying to certain nonlinear Schr\"odinger equations, and it may be interesting to extend this approach to \eqref{pvd} in higher dimensions. We do not pursue this here.

\subsubsection*{Organisation of the paper.} In Section \ref{sec:mainpf} we give a proof of Theorems \ref{mainagain} and \ref{mainagain'} (thus also proving Theorems \ref{main} and \ref{main'}).  Finally, in Sections \ref{sec:var} and \ref{sec:varp} we establish a version of Theorem \ref{main} in the context of more general oscillatory integral operators.

\subsubsection*{Acknowledgments} We thank Neal Bez, Tony Carbery, Taryn Flock, Susana Guti\'errez and Alessio Martini for many helpful discussions surrounding this work.

\section{The proof of Theorem \ref{mainagain}}\label{sec:mainpf} The proof we present follows the same lines as the classical case $n=2$: a suitable change of variables that allows the multilinear extension operator to be expressed as a Fourier transform, followed by Plancherel's theorem.

We have
$$
T_\sigma(g_1,\hdots,g_n)(x_1,\hdots,x_n)=\int_{(\mathbb{R}^{n-1})^n}e^{i(x_1'\cdot\xi_1+\cdots +x_n'\cdot\xi_n)}e^{i(x_{1n}\phi_1(\xi_1)+\cdots +x_{nn}\phi_n(\xi_n))}G(\xi)d\xi,
$$
where $x_j=(x_j',x_{jn})\in\mathbb{R}^{n-1}\times\mathbb{R}$, for each $j$, and
$$
G(\xi):=
\left|\det\left(
\begin{array}{ccc}
1& \cdots  & 1\\
\nabla\phi_1(\xi_1) &
\cdots & \nabla\phi_n(\xi_n)\\
\end{array}\right)\right|^\sigma
g_1(\xi_1)\cdots g_n(\xi_n).
$$
On the subspace $x_1+\cdots +x_n=0$ we therefore have
\begin{eqnarray*}
\begin{aligned}
E_1g_1(x_1)\cdots E_ng_n(x_n)&
=E_1g_1(x_1)\cdots E_ng_n(-x_1-\cdots-x_{n-1})\\
&=\int_{(\mathbb{R}^{n-1})^n}e^{i(x_1'\cdot(\xi_1-\xi_n)+\cdots +x_{n-1}'\cdot(\xi_{n-1}-\xi_n))}\\&\times e^{i(x_{1n}(\phi_1(\xi_1)-\phi_n(\xi_n))+\cdots +x_{(n-1)n}(\phi_{n-1}(\xi_{n-1})-\phi_n(\xi_n)))}G(\xi)d\xi.
\end{aligned}
\end{eqnarray*}
We now make the change of variables $\eta_j=\xi_j-\xi_n$ for each $1\leq j\leq n-1$, so that
\begin{eqnarray*}
\begin{aligned}
E_1g_1(x_1)&\cdots E_ng_n(-x_1-\cdots-x_{n-1})\\
&=\int_{(\mathbb{R}^{n-1})^n}e^{i(x_1'\cdot\eta_1+\cdots +x_{n-1}'\cdot\eta_{n-1})}\\&\times e^{i(x_{1n}(\phi_1(\eta_1+\xi_n)-\phi_n(\xi_n))+\cdots +x_{(n-1)n}(\phi_{n-1}(\eta_{n-1}+\xi_n)-\phi_n(\xi_n)))}\\
&\times G(\eta_1+\xi_n,\hdots,\eta_{n-1}+\xi_n,\xi_n)d\eta_1\cdots d\eta_{n-1}d\xi_n .
\end{aligned}
\end{eqnarray*}
Applying Plancherel's theorem in the variables $x_1',\hdots,x_{n-1}'$ gives
\begin{eqnarray}\label{place1}
\begin{aligned}
\int_{x_1+\cdots +x_n=0}&|E_1g_1(x_1)|^2\cdots |E_ng_n(x_n)|^2dx\\=
& (2\pi)^{(n-1)^2}\int\Bigl|\int e^{i(x_{1n}(\phi_1(\eta_1+\xi_n)-\phi_n(\xi_n))+\cdots +x_{(n-1)n}(\phi_{n-1}(\eta_{n-1}+\xi_n)-\phi_n(\xi_n)))}\\&\times G(\eta_1+\xi_n,\hdots,\eta_{n-1}+\xi_n,\xi_n)d\xi_n\Bigr|^2d\eta_1\cdots d\eta_{n-1}dx_{1n}\cdots dx_{(n-1)n}.
\end{aligned}
\end{eqnarray}
For fixed $\eta_1,\hdots,\eta_{n-1}$ we make the change of variables $\xi_n\mapsto t$, where $t_j=\phi_j(\eta_j+\xi_n)-\phi_n(\xi_n)$ for $1\leq j\leq n-1$. This map is injective on the support of $G_{\eta}:=G(\eta_1+\;\cdot\;,\ldots,\eta_{n-1}+\;\cdot\;,\;\cdot\;)$. Indeed, if not, then there exist $\xi_1\neq \xi_2$ in the support of $G_{\eta}$ (implying that $\eta_j+\xi_1, \eta_j+\xi_2$ are both in the support of $g_j$, for all $j=1,\ldots,n-1$), such that
$$\phi_j(\eta_j+\xi_1)-\phi_n(\xi_1)=\phi_j(\eta_j+\xi_2)-\phi_n(\xi_2)\text{ for all }j=1,\ldots,n-1,
$$
i.e. such that
$$\phi_1(\eta_1+\xi_1)-\phi_1(\eta_1+\xi_2)=\ldots=\phi_{n-1}(\eta_{n-1}+\xi_1)-\phi_{n-1}(\eta_{n-1}+\xi_2)=\phi_n(\xi_1)-\phi_n(\xi_2).
$$
Note that, for all $j=1,\ldots,n-1$, the line segment $\ell_j$ connecting $\eta_j+\xi_1$ with $\eta_j+\xi_2$ is just a parallel translate of the line segment $\ell_n$ connecting $\xi_1$ with $\xi_2$. Of course, for all $j=1,\ldots,n$, $\ell_j$ is contained in the convex hull of the support of $g_j$, and so by our hypotheses, the determinant in the statement of Theorem \ref{main} is non-zero whenever $\xi_j\in\ell_j$ for all $1\leq j\leq n$ . By the mean value theorem for each $\phi_j$ on the line segment $\ell_j$, it follows that, for all $j=1,\ldots,n$, there exists $c_j\in\ell_j$,  such that the directional derivative of $\phi_j$ at $c_j$, in direction $\xi_1-\xi_2$, has the same value for all $j$. In other words,
$$\nabla \phi_j(c_j)\cdot (\xi_1-\xi_2)=c\text{ for all }j=1,\ldots,n,
$$
for some constant $c\in\R$. Therefore,
$$\left(1,\nabla \phi_j(c_j)\right)\cdot (-c,\xi_1-\xi_2)=0\text{ for all }j=1,\ldots,n.
$$
Since $c_j\in\ell_j$ for all $j$, the vectors $\left(1,\nabla \phi_j(c_j)\right)$, $j=1,\ldots,n$, span $\R^n$; thus
$$(-c,\xi_1-\xi_2)=0,
$$
which is a contradiction, since $\xi_1\neq \xi_2$. Therefore, our map is injective. Moreover, the Jacobian determinant of the transformation $\xi_n\mapsto t$ is simply
$$
\frac{\partial t}{\partial\xi_n}=\left|\det\left(
\begin{array}{cccc}
1& \cdots  & 1 & 1\\
\nabla\phi_1(\eta_1+\xi_n) &
\cdots & \nabla\phi_{n-1}(\eta_{n-1}+\xi_n) & \nabla\phi_n(\xi_n)\\
\end{array}\right)\right|,
$$
which does not vanish on the support of $G$.
It follows that
\begin{eqnarray*}
\begin{aligned}
\int_{x_1+\cdots +x_n=0}&|E_1g_1(x_1)|^2\cdots |E_ng_n(x_n)|^2dx\\=
& (2\pi)^{(n-1)^2}\int\Bigl|\int e^{i(t_1x_{1n}+\cdots +t_{n-1}x_{(n-1)n})}\\&\times G(\eta_1+\xi_n,\hdots,\eta_{n-1}+\xi_n,\xi_n)\left(\frac{\partial t}{\partial\xi_n}\right)^{-1}dt\Bigr|^2dx_{1n}\cdots dx_{(n-1)n} d\eta_1\cdots d\eta_{n-1},
\end{aligned}
\end{eqnarray*}
which by Plancherel's theorem again, becomes
\begin{eqnarray*}
\begin{aligned}
(2\pi)^{n(n-1)}\int\left|G(\eta_1+\xi_n,\hdots,\eta_{n-1}+\xi_n,\xi_n)\left(\frac{\partial t}{\partial\xi_n}\right)^{-1}\right|^2dt \; d\eta_1\cdots d\eta_{n-1}.
\end{aligned}
\end{eqnarray*}
Undoing both of the changes of variables above, this expression becomes
\begin{eqnarray*}
\begin{aligned}
(2\pi)^{n(n-1)}\int|G(\xi_1,&\hdots,\xi_{n-1},\xi_n)|^2\left|\left(\frac{\partial t}{\partial\xi_n}\right)\right|^{-1}d\xi\\&=(2\pi)^{n(n-1)}\int\frac{|g_1(\xi_1)|^2\cdots|g_n(\xi_n)|^2}{\left|\det\left(
\begin{array}{ccc}
1& \cdots  & 1\\
\nabla\phi_1(\xi_1) &
\cdots & \nabla\phi_n(\xi_n)\\
\end{array}\right)\right|^{1-2\sigma}}d\xi,
\end{aligned}
\end{eqnarray*}
as claimed.

\section{The proof of Theorem \ref{mainagain'}}

Following the proof of Theorem \ref{mainagain}, we reach \eqref{place1} and apply the same change of variables $\xi_n\mapsto t$, which, in this case, is explicitly given by
$$t_j=\phi_j(\eta_j+\xi_n)-\phi_n(\xi_n)=|\eta_j|^2+2\eta_j\cdot\xi_n.
$$
For every $(\eta_1,\ldots,\eta_{n-1})\in\left(\R^{n-1}\right)^{n-1}$ that span $\R^{n-1}$ (that is, for almost every $(\eta_1,\ldots,\eta_{n-1})$), the above affine transformation is globally injective, with Jacobian determinant
$$2^{n-1}\eta_1\wedge\ldots\wedge\eta_{n-1}=\det\left(
\begin{array}{cccc}
1& \cdots  & 1 & 1\\
\nabla\phi_1(\eta_1+\xi_n) &
\cdots & \nabla\phi_{n-1}(\eta_{n-1}+\xi_n) & \nabla\phi_n(\xi_n)\\
\end{array}\right)\not=0.
$$
The proof now concludes as in the proof of Theorem \ref{mainagain}.

\section{Variable coefficient generalisations}\label{sec:var} It is natural to attempt to generalise Theorem \ref{main} (at the level of an inequality) to encompass families of more general oscillatory integral operators of the form
$$
T_\lambda f(x)=\int_{\mathbb{R}^{n-1}}e^{i\lambda\Phi(x,\xi)}\psi(x,\xi)f(\xi)d\xi,
$$
where $\Phi$ is a smooth real-valued phase function, $\psi$ is a compactly-supported bump function, and $\lambda$ is a large real parameter.
%The phase $\Phi$ here will be generalisation of $x\cdot\Sigma(\xi)$, where $\Sigma(\xi)=(\xi,\phi(\xi))$ from earlier on, and $\psi$ is a compactly supported smooth function.
%This is effectively done for $n=2$ in \cite{S} (Page 412), and is a famous argument often attributed to Carleson and Sj\"olin.

To this end, suppose that we have $n$ of these operators, $T_{1,\lambda},\hdots,T_{n,\lambda}$ with phases $\Phi_1,\hdots,\Phi_n$ (and cutoff functions $\psi_1,\hdots,\psi_n$).
An appropriate transversality condition is that the kernels of the mappings $d_\xi d_x\Phi_1,\hdots,d_\xi d_x\Phi_n$ span $\mathbb{R}^n$ at every point. In order to be more precise let
$$
X(\Phi_j):=\bigwedge_{\ell=1}^{n-1}\frac{\partial}{\partial
\xi_{\ell}}\nabla_{x}\Phi_j
$$
for each $1\leq j\leq n$; by (Hodge) duality we may interpret each $X(\Phi_j)$ as an $\mathbb{R}^n$-valued function on $\mathbb{R}^n\times\mathbb{R}^{n-1}$.
In the extension case where $\Phi_j(x,\xi)=x\cdot\Sigma_j(\xi)$, observe that $X(\Phi_j)(x,\xi)$ is simply a vector normal to the surface $S_j$ at the point $\Sigma_j(\xi)$. A natural transversality condition to impose on the general phases $\Phi_1,\hdots,\Phi_n$ is thus
\begin{equation}\label{difftrans}
\det\left(X(\Phi_{1})(x_1,\xi_1),\hdots,X(\Phi_{n})(x_n,\xi_n)\right)\gtrsim 1
\end{equation}
for all
$(x_1,\xi_1)\in\supp(\psi_{1}),\hdots,(x_n,\xi_n)\in\supp(\psi_{n})$. Under this condition it is shown in \cite{BCT} that
\begin{equation}\label{bctvar}
\Bigl\|\prod_{j=1}^{n}T_{j,\lambda}f_{j}\Bigl\|_{L^{\frac{2}{n-1}}(\mathbb{R}^{n})} \leq
C_{\varepsilon}\lambda^{-\frac{n(n-1)}{2}+\varepsilon}\prod_{j=1}^{n}\|f_{j}\|_{L^{2}(\mathbb{R}^{n-1})},
\end{equation}
generalising \eqref{BCTthm}. Here we establish the corresponding generalisation of \eqref{alttran}.
\begin{theorem}\label{varconj}
Assuming \eqref{difftrans}
\begin{equation}\label{varalttran}
\int_{x_1+\cdots +x_n=0}|T_{1,\lambda}f_1(x_1)|^2\cdots |T_{n,\lambda}f_n(x_n)|^2dx\lesssim\lambda^{-n(n-1)}\|f_1\|_2^2\cdots\|f_n\|_2^2.
\end{equation}
\end{theorem}
Of course \eqref{bctvar} with $\varepsilon=0$ is the same as \eqref{varalttran} when $n=2$.
Theorem \ref{varconj} is well-known for $n=2$, and this is a simple exercise using H\"ormander's theorem for nondegenerate oscillatory integral operators. More precisely, observe that when $n=2$,
$$
T_{1,\lambda}f_1(x)T_{2,\lambda}f_2(-x)=\int_{(\mathbb{R})^2}e^{i\lambda\Psi(x,\xi)}\psi_1(x,\xi_1)\psi_2(x,\xi_2)f_1(\xi_1)f_2(\xi_2)d\xi_1d\xi_2,
$$
where $\Psi(x,\xi):=\Phi_1(x,\xi_1)+\Phi_2(-x,\xi_2)$, and notice that
$\det\hess\Psi$ coincides with the nonzero quantity in the hypothesis \eqref{difftrans}. Hence \eqref{varalttran} holds for $n=2$ by H\"ormander's theorem; see \cite{CS} and \cite{S} for further context and discussion.
As may be expected from Section \ref{sec:mainpf}, the higher-dimensional case of Theorem \ref{varconj} will follow by a similar argument, although some additional linear-algebraic ingredients will be required.

\section{Proof of Theorem \ref{varconj}}\label{sec:varp}
We begin by writing
\begin{eqnarray*}
\begin{aligned}
T_{1,\lambda}f_1(x_1)\cdots & T_{n-1,\lambda}f_{n-1}(x_{n-1})T_{n,\lambda}f_n(-x_1-\cdots-x_{n-1})\\&=\int_{(\mathbb{R}^{n-1})^n}e^{i\lambda\Psi(x,\xi)}\psi_1(x,\xi_1)\cdots\psi_n(x,\xi_n)f_1(\xi_1)\cdots f_n(\xi_n)d\xi,
\end{aligned}
\end{eqnarray*}
where
$
\Psi:(\mathbb{R}^n)^{n-1}\times(\mathbb{R}^{n-1})^n\rightarrow\mathbb{R}
$
is given by
\begin{equation}\label{defPsi}
\Psi(x,\xi)=\Phi_1(x_1,\xi_1)+\cdots+\Phi_{n-1}(x_{n-1},\xi_{n-1})+\Phi_n(-x_1-\cdots-x_{n-1},\xi_n).
\end{equation}
The difficulty now is that $\hess\Psi$ is no longer an $n\times n$ matrix, and so some work has to be done to see that its determinant coincides with that in the hypothesis \eqref{difftrans}. Once this is done Theorem \ref{varconj} follows by a direct application of H\"ormander's theorem as in the case $n=2$.
Thus matters are reduced to showing the following.

\begin{proposition}\label{hess}$$\det\hess\Psi(x,\xi)=(-1)^{n-1+\frac{(n-1)^2(n-2)}{2}}\det\big(X(\Phi_{1})(x_1,\xi_1)\hdots,X(\Phi_{n})(-x_1-\cdots-x_{n-1},\xi_n)\big);$$
the coefficient above equals $1$ for $n\equiv 0,1,3\;({\rm mod}\;4)$, and $-1$ for $n\equiv 2\;({\rm mod}\;4)$.
\end{proposition}

Note that $\hess\Psi(x,\xi)$ is of the form
$$\left( \begin{array}{cccccc}
A^{(1)}_{n\times(n-1)} & \mathbf{0} & \mathbf{0} & \ldots & \mathbf{0} & A^{(n)}_{n\times (n-1)} \vspace{0.1in}\\
\mathbf{0} & A^{(2)}_{n\times (n-1)} & \mathbf{0} &\ldots & \mathbf{0} & A^{(n)}_{n\times (n-1)} \vspace{0.1in}\\
\mathbf{0} & \mathbf{0} & A^{(3)}_{n\times (n-1)} &\ldots & \mathbf{0} & A^{(n)}_{n\times (n-1)} \vspace{0.1in}\\
\vdots & \vdots & & \ddots & \vdots & \vdots \\
\mathbf{0} & \mathbf{0} & \ldots & \mathbf{0} & A^{(n-1)}_{n\times (n-1)} & A^{(n)}_{n\times (n-1)} \\
\end{array}\right),
$$
where $$A^{(i)}_{n\times (n-1)}={\rm Hess}\;\Phi_{i}(x_i,\xi_i)\text{ for }i=1,\ldots, n-1$$ and $$A^{(n)}_{n\times (n-1)}=-{\rm Hess}\;\Phi_n({-x_1-\cdots-x_{n-1}, \xi_n}).$$
Proposition \ref{hess} is therefore a special case of Lemma \ref{generaldet} that follows, for $k=n-1$.

\begin{lemma}\label{generaldet} For $n\geq 2$ and $1\leq k\leq n-1$, let $A_{n\times (n-1)}^{(1)},\ldots,A_{n\times (n-1)}^{(k)} $ be $n\times (n-1)$-block matrices, and $A_{n\times k}^{(k+1)}$ be an $n\times k$-block matrix. Let
$$
M_{n,k}:=
\left(
\begin{array}{cccccc}
A^{(1)}_{n\times(n-1)} & \mathbf{0} & \mathbf{0} & \ldots & \mathbf{0} & A^{(k+1)}_{n\times k} \vspace{0.1in}\\
\mathbf{0} & A^{(2)}_{n\times (n-1)} & \mathbf{0} &\ldots & \mathbf{0} & A^{(k+1)}_{n\times k} \vspace{0.1in}\\
\mathbf{0} & \mathbf{0} & A^{(3)}_{n\times (n-1)} &\ldots & \mathbf{0} & A^{(k+1)}_{n\times k} \\
\vdots & \vdots & & \ddots & \vdots & \vdots \vspace{0.1in}\\
\mathbf{0} & \mathbf{0} & \ldots & \mathbf{0} & A^{(k)}_{n\times (n-1)} & A^{(k+1)}_{n\times k} \\
\end{array}\right).$$
Then,
$$\det M_{n,k}=(-1)^{(n-1)\frac{k(k-1)}{2}}\Lambda^*_1 \wedge \cdots \wedge \Lambda_{k}^* \wedge \Lambda^*_{k+1},
$$
where $\Lambda^*_{i}$ is the (Hodge) dual of the wedge product $\Lambda_{i}$ of the columns of $A^{(i)}_{n\times (n-1)}$, for all  $i=1,\ldots, k$, and $\Lambda_{k+1}^*$is the dual of the wedge product $\Lambda_{k+1}$ of the columns of $A^{(k+1)}_{n\times k}$.
\end{lemma}

\begin{proof} For any $1\leq k \leq n-1$, we denote by $C_i$ the $i$-th column of $A^{(k+1)}_{n\times k}$. By definition,\\ \\
$\;\Lambda^*_1 \wedge \ldots \wedge \Lambda_{k}^* \wedge \Lambda^*_{k+1}=
$
$$=\det\left(
\begin{array}{cccc}
\langle\Lambda_1^*, C_1\rangle & \langle\Lambda_1^*, C_2\rangle & \ldots & \langle\Lambda_1^*, C_k\rangle  \vspace{0.1in}\\
\langle\Lambda_2^*, C_1\rangle & \langle\Lambda_2^*, C_2\rangle & \ldots & \langle\Lambda_2^*, C_k\rangle \\
 & \;\;\;\;\;\;\;\;\;\;\;\;\;\;\;\;\;\;\;\vdots& & \\
\langle\Lambda_k^*, C_1\rangle & \langle\Lambda_k^*, C_2\rangle & \ldots & \langle\Lambda_k^*, C_k\rangle \\
\end{array}\right)
$$
$$=\det\left(
\begin{array}{cccc}
\Lambda_1\wedge C_1 & \Lambda_1\wedge C_2 & \ldots & \Lambda_1\wedge C_k \\
\Lambda_2\wedge C_1 & \Lambda_2\wedge C_2 & \ldots & \Lambda_2\wedge C_k \\
 & \;\;\;\;\;\;\;\;\;\;\;\;\;\;\;\;\;\;\;\vdots& & \\
\Lambda_k\wedge C_1 & \Lambda_k\wedge C_2 & \ldots & \Lambda_k\wedge C_k \\
\end{array}\right).
$$

It thus suffices to show that, for any $n\geq 2$ and $k\leq n$, \begin{equation}\label{eq:wedgedet}\det M_{n,k}=
(-1)^{(n-1)\frac{k(k-1)}{2}}\det\left(
\begin{array}{cccc}
\Lambda_1\wedge C_1 & \Lambda_1\wedge C_2 & \ldots & \Lambda_1\wedge C_k \\
\Lambda_2\wedge C_1 & \Lambda_2\wedge C_2 & \ldots & \Lambda_2\wedge C_k \\
 & \;\;\;\;\;\;\;\;\;\;\;\;\;\;\;\;\;\;\;\vdots& & \\
\Lambda_k\wedge C_1 & \Lambda_k\wedge C_2 & \ldots & \Lambda_k\wedge C_k \\
\end{array}\right).
\end{equation}
We prove \eqref{eq:wedgedet} by induction on $k$.\\ \\
Indeed, \eqref{eq:wedgedet} clearly holds for $k=1$; in that case,
$$\det M_{n,1}=\det\left(
\begin{array}{cc}
A^{(1)}_{n\times(n-1)} & C_1 \\
\end{array}\right)=\Lambda_1\wedge C_1.
$$
Let $k\geq 2$, and let us assume that \eqref{eq:wedgedet} holds for $k-1$; we now deduce it for $k$. We first observe that $$\det M_{n,k}=\sum_{i=1}^{k}\det B_i,
$$
where
$$B_i=\left(
\begin{array}{cccccc}
A^{(1)}_{n\times(n-1)} & \mathbf{0} & \mathbf{0} & \ldots & \mathbf{0} & 0 \;\;\;\;\;\; \ldots \;\;\;\;\;\; 0 \;\;\;\;\;\; C_i\;\;\;\;\;\;0\;\;\;\;\;\; 0 \\
\mathbf{0} & A^{(2)}_{n\times (n-1)} & \mathbf{0} &\ldots & \mathbf{0} & \;\;C_1 \;\;\;\; \ldots \;\;\;\; C_{i-1} \;\;\;\; 0\;\;\;\;C_{i+1}\;\;\;\; C_k \\
\mathbf{0} & \mathbf{0} & A^{(3)}_{n\times (n-1)} &\ldots & \mathbf{0} & \;\;C_1 \;\;\;\; \ldots \;\;\;\; C_{i-1} \;\;\;\; 0\;\;\;\;C_{i+1}\;\;\;\; C_k \\
\vdots & \vdots & & \ddots & \vdots & \vdots \\
\mathbf{0} & \mathbf{0} & \ldots & \mathbf{0} & A^{(k)}_{n\times (n-1)} & \;\;C_1 \;\;\;\; \ldots \;\;\;\; C_{i-1} \;\;\;\; 0\;\;\;\;C_{i+1}\;\;\;\; C_k\\
\end{array}\right).$$
Indeed, let us focus on the last $k$ columns of $M_{n,k}$. By writing the $i$-th of these columns in the form $(C_i,0,\hdots,0)+(0,C_i,\ldots,C_i)$, for all $i=1,\hdots,k$, multilinearity of the determinant implies that
$$\det M_{n,k}=\sum_{i=1}^{k}\det B_i + \sum_{i\neq j}\det \Gamma_{i,j},
$$
where $\Gamma_{i,j}$ is an $nk\times nk$ matrix with $(C_i,0,\hdots,0)$ and $(C_j,0,\hdots,0)$ as the $i$-th and $j$-th column of its right $nk\times k$ block. These columns, together with the columns of $A^{(1)}_{n\times (n-1)}$, form a set of $n+1$ vectors in $\R^{n-1}$, and are thus linearly dependent, forcing the determinant of $\Gamma_{i,j}$ to be zero.

We now swap the column $(C_i, 0, \hdots, 0)$ consecutively with columns on its immediate left until it becomes the $n$-th column; there are $i-1+(n-1)(k-1)$ such swaps involved, therefore
\begin{equation} \label{eq:determinant}\det M_{n,k}=(-1)^{(n-1)(k-1)}\sum_{i=1}^{k}(-1)^{i-1}\det D_i,
\end{equation}
where $D_i$ is the matrix we get from $B_i$ by the above process; in other words,
$$D_i=\left(
\begin{array}{cccccc}
A^{(1)}_{n\times(n-1)}\;\;\;\; C_i\;\;\; & \mathbf{0} & \mathbf{0} & \ldots & \mathbf{0} & \mathbf{0} \\
\mathbf{0} & A^{(2)}_{n\times (n-1)} & \mathbf{0} &\ldots & \mathbf{0} & \widehat{A}^i_{n, k-1} \\
\mathbf{0} & \mathbf{0} & A^{(3)}_{n\times (n-1)} &\ldots & \mathbf{0} & \widehat{A}^i_{n, k-1} \\
\vdots & \vdots & & \ddots & \vdots & \vdots \\
\mathbf{0} & \mathbf{0} & \ldots & \mathbf{0} & A^{(k)}_{n\times (n-1)} & \widehat{A}^i_{n, k-1} \\
\end{array}\right),
$$
where $\widehat{A}^i_{n, k-1}$ denotes the $n\times (k-1)$ matrix that we get from $A_{n\times k}^{(k+1)}$ after deleting its $i$-th column. Since $\left(
\begin{array}{cc}
A^{(1)}_{n\times(n-1)} & C_i\\
\end{array}\right)
$
is a square matrix, we obtain
$$\det D_i=\left(\Lambda_1\wedge C_i\right) \cdot \det \left(
\begin{array}{ccccc}
 A^{(2)}_{n\times (n-1)} & \mathbf{0} &\ldots & \mathbf{0} & \widehat{A}^i_{n, k-1} \\
 \mathbf{0} & A^{(3)}_{n\times (n-1)} &\ldots & \mathbf{0} & \widehat{A}^i_{n, k-1} \\
 \vdots & & \ddots & \vdots & \vdots \\
 \mathbf{0} & \ldots & \mathbf{0} & A^{(k)}_{n\times (n-1)} & \widehat{A}^i_{n, k-1} \\
\end{array}\right)$$
$$=(-1)^{(n-1)\frac{(k-1)(k-2)}{2}}\left(\Lambda_1\wedge C_i\right)\;\cdot \;\det\left(
\begin{array}{cccccc}
\Lambda_2\wedge C_1 & \ldots & \Lambda_2\wedge C_{i-1} &
\Lambda_2\wedge C_{i+1} &\ldots & \Lambda_2\wedge C_k \\
\Lambda_3\wedge C_1 & \ldots & \Lambda_3\wedge C_{i-1} &
\Lambda_3\wedge C_{i+1} &\ldots & \Lambda_3\wedge C_k \\
 & &\;\;\;\;\;\;\;\;\;\;\;\;\;\;\;\;\;\;\;\vdots& & & \\
\Lambda_k\wedge C_1 & \ldots & \Lambda_k\wedge C_{i-1} &
\Lambda_k\wedge C_{i+1} &\ldots & \Lambda_k\wedge C_k \\
\end{array}\right);
$$
the last equality holds by the inductive hypothesis. Plugging this into \eqref{eq:determinant}, we obtain \eqref{eq:wedgedet} for this $k$.

\end{proof}

\end{document}